\documentclass[11pt]{amsart}
\usepackage{ae}
\usepackage[applemac]{inputenc}
\usepackage{comment}
\usepackage{amsmath,amsthm,amsfonts, mathtools} 
\usepackage{enumerate,bbm}

\theoremstyle{plain}

\newtheorem{Theorem}{Theorem}[section]
\newtheorem{theorem}[Theorem]{Theorem}

\newtheorem{Lemma}[Theorem]{Lemma}
\newtheorem{lemma}[Theorem]{Lemma}

\newtheorem{proposition}[Theorem]{Proposition}
\theoremstyle{definition}
\theoremstyle{definitionsAsText}

\theoremstyle{definition}%

\theoremstyle{remark}

\newtheorem*{remarks}{Remarks}

\newcommand{\textita}[1]{\textit{#1\/}}

\newcommand{\prova}{\begin{proof}}	
\newcommand{\fimprova}{\end{proof}}

\newcommand{\lista}{\renewcommand{\theenumi}{\alph{enumi}}\begin{listalf}}
\newcommand{\fimlista}{\end{listalf}}

\newcommand{\listr}{\makeatletter\renewcommand{\theenumi}{\roman{enumi}}
\renewcommand{\labelenumi}{(\theenumi)}\makeatother
\begin{enumerate}}
\newcommand{\fimlistr}{\end{enumerate}}

\newenvironment{listrom}{\makeatletter\renewcommand{\theenumi}
{\roman{enumi}
}
\renewcommand{\labelenumi}{\textnormal{(\theenumi\!\!)}}\makeatother
\begin{enumerate}}{\end{enumerate}
}

\newenvironment{listalf}{\makeatletter\renewcommand{\theenumi}
{\alph{enumi}
}
\renewcommand{\labelenumi}{\textnormal{(\theenumi\!\!)}}\makeatother
\begin{enumerate}}{\end{enumerate}
}

\newcounter{listromSinC}

\newcounter{listromSSinC}

\newcounter{enumerateSinC}

\newcommand{\intern}[1]{\langle #1\rangle}
 \newcommand{\fecho}[1]{\overline{#1}}

\newcommand{\sspan}{\operatorname{span}}

 \newcommand{\AAA}{{\mathcal{ T}(\mathcal{N})}}
\newcommand{\BBB}{{ \mathcal{B}}}

\newcommand{\NNN}{{ \mathcal{N}}}
\newcommand{\XXX}{{ \mathfrak{X}}}

 \newcommand{\HHH}{{ \mathfrak{H}}}

\newcommand{\LLL}{{ \mathfrak{L}}}

\newcommand{\MMM}{{ \mathcal{M}}}

\newcommand{\SSS}{{ \mathfrak{S}}}

\renewcommand{\mathbb}[1]{\mathbbm{#1}}
\begin{document}
\title[Finite rank operators in Lie ideals]{Finite rank operators in Lie ideals of 
 nest algebras}

\author[L. Oliveira ]{Lina Oliveira}
\address{Department of Mathematics\\Instituto Superior 
T\'ecnico\\
Av. Rovisco Pais\\
1049-001 Lisbon\\
Portugal}
\email{linaoliv@math.ist.utl.pt}

\begin{abstract}The main theorem  provides a characterisation of the finite rank operators lying in a norm closed Lie ideal of a continuous nest algebra.
These operators 
 are charaterised as those  finite rank operators in the nest  algebra satisfying a condition determined by a left order continuous homomorphism on the nest. 
 A crucial fact used in the proof of  this theorem is the decomposability of the finite rank operators.
  One shows  that a finite rank operator in a norm closed Lie ideal of a continuous nest algebra can be written as a finite sum of rank one operators lying in the ideal. 
\end{abstract}

\maketitle
\section{Introduction}\label{intro}
The structure of Lie ideals of nest  algebras 
 has been investigated by many authors  for at least a decade (cf. \cite{DDH, HP,  HMS, MS, LuYu} and the literature referenced therein). Two main lines of this research might be essentially described as  focusing either  on the connection between Lie ideals and associative ideals   
or  on 
 similarity invariant subspaces.

The present work approaches the investigation of Lie ideals from a different perspective, its  departing point being the assumption that the set of finite rank operators in the ideal should be worth investigating. This seems only natural  
if  one bears in mind the important rôle played by the finite rank operators in the theory of nest algebras given their density properties  (cf. \cite{KD, Er}). Furthermore, investigating the decomposability  of the finite rank operators in   Lie ideals  seems to be an interesting problem in its own right. One has simply to refer to the literature to realise that an  effort has been made to investigate the decomposability of finite rank operators lying in various subalgebras  of $\BBB(\HHH)$ (see, for example,  \cite{B, DL, D,  EP} and also \cite{Er}, citing a theorem of Ringrose).

The main result of this work, Theorem \ref{homomorphism finite rank}, appears in Section \ref{character}. This theorem provides a characterisation of the 
finite rank operators lying in a norm closed Lie ideal $\LLL$ of a continuous nest algebra $\AAA$. This is a characterisation along the lines of that obtained by Erdos and Power  for weakly closed associative ideals (cf. \cite{EP}, Theorem 1.5) and characterises the finite rank operators lying in $\LLL$ as those in the nest  algebra satisfying a condition determined by a left order continuous homomorphism on the nest $\NNN$.

A crucial fact used in the proof of Theorem \ref{homomorphism finite rank} is that a finite rank operator in a norm closed Lie ideal of a continuous nest algebra can be written as a finite sum of rank one operators lying in the ideal. The decomposability of the finite rank operators is stated in Theorem \ref{finite rank}, which is the principal result of Section \ref{finite rank op}. The remaining contents of this section consists in the statement and proof of  auxiliary results required  to prove  Theorem \ref{finite rank}.

Of the four sections of which this work consists, only Section 2 remains to be described. This is a preliminary section establishing notation and recalling well-known facts about nest algebras needed in the sequel.

\section{Preliminaries}\label{prelim}

Let $\HHH$ be a complex Hilbert space and let $\BBB(\HHH)$ be the complex Banach algebra of bounded linear operators on $\HHH$. 
A totally ordered family $\NNN  $ of projections in $\BBB(\HHH)$ containing $0 $ and the identity $I $
is said to be a \textit{nest}. If, furthermore, $\NNN  $ is a complete sublattice
of the lattice of projections in $\BBB(\HHH)$,  then $\NNN $ is called a \emph{complete 
nest}. 

  Let $P$ be a projection in the nest $\NNN$ and  define $P_{-}$ to be the projection in $\NNN$ satisfying $$P_{-} =\bigvee \{Q\in \NNN: Q<P\},$$ if $P$ is not zero, and $P_-$ is zero, otherwise.  
  A complete nest $\NNN$ is said to be \emph{continuous} if, for all 
projections $P$ 
in $\NNN$, the projections $P$ and $P_{-}$ coincide.

The \emph{nest 
algebra} $\AAA$ associated 
with a nest $\NNN  $ is the subalgebra of all operators $T $ in $\BBB  (\HHH )     $ such that, for all projections $P $ 
in $\NNN  $, 
 \[T(P  (\HHH ))
\subseteq P  (\HHH ), 
\]
or, equivalently,
an operator $T $ in $\BBB  (\HHH ) $ lies in $\AAA  $ if and only 
if, for all projections $P $ in the nest $\NNN  $,
 \[   
P^\perp  T P = 0.
\] 
It is well-known that $\AAA  $ is a unital 
weak operator closed subalgebra of
$\BBB  (\HHH ) $  and  
  that each nest is 
contained in a complete nest which generates the same nest algebra 
(cf. \cite{KD, Ring}).  Since this work
is mainly concerned   with the nest algebras and not with the nests 
themselves, henceforth only complete nests will be considered. A nest algebra associated to a continuous nest is said to be a \emph{continuous} nest algebra.

A nest algebra $\AAA$ together with the product, defined for all operators $T$ and $S$ in $\AAA$, by
$$
[T,S]=TS-ST
$$
is a Lie algebra and    a complex subspace $\LLL$ of the nest algebra $\AAA$
is said to be a \emph{Lie ideal}  if $[\LLL, \AAA]\subseteq \LLL$.

Let  $x$ and $y$ be elements of the Hilbert space  $\HHH$ and  let $x\otimes y$ be the rank one operator   defined, for all $z$ in $\HHH$, by 
 $$
  z\mapsto \intern{z, x}y,
  $$
   where $ \intern{\cdot, \cdot}$ denotes the inner product of $\HHH$. A rank one operator $x\otimes y$ lies in $\AAA$ if and only if  
   $$
   P_- x=0\quad \quad \quad \text{and} \quad \quad \quad Py=y,
   $$
   for some projection $P$ in the nest, and $P$ can be chosen to be equal to  $\bigwedge\{Q\in \NNN: Qy=y\}$ (cf. \cite{Ring}). As a consequence, since the nest $\NNN$ is continuous,  for any $x\otimes y$  in $\AAA$, the elements $x$ and $y$ are necessarily orthogonal. 
  (For  the general theory  of nest algebras, the reader is referred to \cite{KD, Ring}.)

%%%%%%%%%%%%%%%%%%%%%%%%%%%%%%%%%%%%%%%%%%%%%%%%%%%%%%%%  SECTION FINITE RANK %%%%%%%%%%%%%%%%%%%%%%%%%%%%%%%%%%%%%%%%%%%%%%%%%%%%%%%%%%%%%

\section{Finite rank operators}\label{finite rank op}
 
 The principal result of this section, Theorem \ref{finite rank},  
establishes the possibility of 
constructing the finite rank operators in  $\LLL$  as finite sums of rank one operators also lying in the ideal. 
%%%%%%%%%%%%%%%%oliveira_arxiv2%%%%%%%%%%%%%%%%%%%%
A central idea underlies the proof of this theorem, and implicitly also the proofs  of some of the auxiliary results. It might be described as regarding each rank one operator $x\otimes y$ in any nest algebra $\AAA$ from the point of view of a particular pair of projections  
 intrinsically associated to $x$ and $y$. The way in which this projections  are built is described as follows.

Let $z$ be an element of the Hilbert space $\HHH$ and, as in \cite{LO},  let  the  projections $P_z$ and $\hat{P}_z$ be defined by
 
 $$
 P_z=\bigwedge\{Q\in \NNN: Qz=z\}\quad  \text{,} \quad \quad \hat{P}_z=\bigvee \{Q\in \NNN: Qz=0\}.
 $$
 The projections $P_z$ and $\hat{P}_z$ lie in the nest $\NNN$ and are such that 
 $$
 P_zz=z \, , \quad \quad  \quad  \quad \quad  \hat{P}_z z=0.
 $$
The pair of projections associated to the rank one operator $x\otimes y$  are $\hat{P}_x$ and $P_y$, respectively. The proof of the next lemma  is a first example of how 
these projections  are used. In the  lemma,  for a given norm closed Lie ideal $\LLL$ of a continuous nest algebra,
  some rank one operators  are singled out as lying necessarily in  $\LLL$, provided that a particular operator $x\otimes y$ lies in $\LLL$. 

%%%%%%%%%%%%%%%%%%%%%%%%%%%%%%%%%%%%%%%%%%%
 \begin{Lemma}\label{1}Let $\AAA$ be a continuous nest algebra, let $\LLL$ be a norm closed Lie ideal of $\AAA$ and let  $x\otimes y$ be a rank one operator lying in  $\LLL$. The following assertions hold.
 \begin{listrom}
 \item  If $x\otimes z$ is a rank one operator in the nest algebra $\AAA$ such that  
 $ P_z\leq P_y$, then  $x\otimes z$ lie in $\LLL$.
 \item  If  $w\otimes y$ is a rank one operator in the nest algebra $\AAA$ such that  
 $ \hat{P}_x\leq \hat{P}_w$, then $w\otimes y$ lies in $\LLL$.
 \end{listrom}
   \end{Lemma}

   \begin{proof}
To prove assertion (i), firstly, suppose that $P_z< P_y$ and  let   $y^{\prime}$ be the non zero element of $\HHH$ defined by
$$
y^{\prime}=P_z^\perp y.
$$ Since  $y^{\prime}\otimes z$ lies in $\AAA$, the bracket 
%\[
 \begin{align*}
[x\otimes y,y^{\prime}\otimes z]&=<z,x>(y^{\prime}\otimes y)-<y,y^{\prime}>(x\otimes z)\\
&=-<y,y^{\prime}>(x\otimes z)=-\|y^{\prime}\|^2(x\otimes z)
\end{align*}
%\]
lies in $\LLL$, yielding that also  $x\otimes z$ lies in $\LLL$.

If $P_z= P_y$, then 
\[
P_y(\HHH)=P_z(\HHH)=\fecho{\bigcup_{P\in \NNN, P< P_z}P(\HHH)},
\]
and, consequently, there exists a sequence $(z_n)$ in 
$$
\bigcup_{P\in \NNN, P< P_z}P(\HHH)
$$
converging to $z$ in the norm topology. It follows that  $(x\otimes z_n)$  is a convergent sequence in  $\LLL$ whose limit $ x\otimes z$ also lies in $\LLL$.

To prove assertion (ii), suppose  initially that  $ \hat{P}_x< \hat{P}_w$.  
Let $z$ be the element of the Hilbert space $\HHH$ defined by 
\[
z=\hat{P}_w x.
\]The operator 
%\[
\begin{align*}
[x\otimes y,w\otimes z]&=<z,x>(w\otimes y)-<y,w>(x\otimes z)\\
&=\|\hat{P}_w x\|^2(w\otimes y),
%\]
\end{align*}
lies in $\LLL$ and, therefore,  also $w\otimes y$ lies in $\LLL$.

%\item 
Suppose now that $\hat{P}_x= \hat{P}_w$.  
Let $(w_n)$ be a sequence in 
$$
\bigcup_{P\in \NNN, \hat{P}_x<P} P^\perp (\HHH)
$$
converging to $w$ in the norm topology. Since, for all $n$, the operator  $w_n\otimes y$ lies in $\LLL$ and the sequence $(w_n\otimes y)$ converges to $w\otimes y$  in the norm topology, it follows that $w\otimes y$ lies in $\LLL$.
\end{proof}

It is now possible to identify a whole ``corner'' of rank one operators  in $\AAA$ as also lying in the Lie  ideal $\LLL$.

 \begin{Theorem}\label{corner}Let $\AAA$ be a continuous nest algebra, let $\LLL$ be a norm closed Lie ideal of $\AAA$, let  $x\otimes y$ be a rank one operator lying in  $\LLL$ and   let $w\otimes z$ be a rank one operator  in the nest algebra $\AAA$ satisfying 
 $$
  \hat{P}_x\leq \hat{P}_w \quad \quad \quad  \text{and} \quad \quad \quad P_z\leq P_y.
  $$
 Then, the operator $w\otimes z$  lies in $\LLL$.
  \end{Theorem}

  \begin{proof}By Lemma \ref{1} (ii), the operator $x\otimes z$ lies in $\LLL$ and applying  (i) of the same lemma to this operator, it follows that $w\otimes z$ lies in  $\LLL$.
\end{proof}

\begin{lemma}\label{finite rank aux1}
Let $\AAA$ be a continuous nest algebra, let $\LLL$ be a norm closed Lie ideal of $\AAA$ and let  
$$
T=\sum_{i=1}^{n}x_i\otimes y_i,
$$
 be a rank $n$ operator  lying in $\LLL$, with $n\geq 1$, where, for all $i, j=1, \dots,n$, the operator $x_i\otimes y_i$ lies in $\AAA$ and $\hat{P}_{x_i} <\hat{P}_{x_j}$, whenever $i<j$.
 Then, for all $i=1, \dots,n$, the rank one operator $x_i\otimes y_i$ lies in $\LLL$.
\end{lemma}

\begin{proof} The assertion trivially holds for $n=1$. Suppose now that   $n$ is greater than 1. The operator 
\begin{align*}
\left[ \hat{P}_{x_n},T \right]  &=\hat{P}_{x_n}\Bigl( \sum_{i=1}^{n}x_i\otimes y_i\Bigr)-\Bigl( \sum_{i=1}^{n}x_i\otimes y_i\Bigr)\hat{P}_{x_n}\\
&= \sum_{i=1}^{n}x_i\otimes(\hat{P}_{x_n} y_i)- \sum_{i=1}^{n}(\hat{P}_{x_n}x_i)\otimes y_i\\
&= \sum_{i=1}^{n}x_i\otimes y_i - \sum_{i=1}^{n-1}(\hat{P}_{x_n}x_i)\otimes y_i\\
&=T- \sum_{i=1}^{n-1} (\hat{P}_{x_n}x_i)\otimes y_i
\end{align*}
lies in $\LLL$. If $n$ coincides with 2, then $(\hat{P}_{x_2}x_1)\otimes y_1$ lies in $\LLL$ and, by Theorem \ref{corner}, the operator $x_1\otimes y_1$ lies in $\LLL$. Hence,   $x_2\otimes y_2$ also lies in $\LLL$ and  thus the proof is complete.  
In the case of $n$ being greater than 2,  
  observe that the subset $\mathfrak{X}$ of the Hilbert space $\HHH$ defined by  
   $$
 \mathfrak{X}=\{\hat{P}_{x_n}x_i:i=1, \dots , n-1\}
 $$
 is linearly independent. In fact, if $\alpha_{1}, \alpha_{2}, \dots, \alpha_{n-1}$ are scalars such that 
\[
\alpha_{1} \hat{P}_{x_n}x_1 +\alpha_{2} \hat{P}_{x_n}x_2 +\dots+\alpha_{n-1} \hat{P}_{x_n}x_ {n-1}=0,
\]
then, since  
$$
\hat{P}_{x_1} <\hat{P}_{x_2}<\dots<\hat{P}_{x_{n-1}}<\hat{P}_{x_n},
$$
it follows that 
\begin{align*}
\hat{P}_{x_2}(\alpha_{1} \hat{P}_{x_n}x_1 +\alpha_{2} \hat{P}_{x_n}x_2 +\dots+\alpha_{n-1} \hat{P}_{x_n}x_ {n-1})&=\alpha_{1}\hat{P}_{x_2} x_1\\
&=0.
\end{align*}
Consequently, $\alpha_{1}$ must be equal to zero since $\hat{P}_{x_2} x_1$ is different from zero. Similarly,  using now  the projection $\hat{P}_{x_3}$ and the equality 
\[
\alpha_{2} \hat{P}_{x_n}x_2 +\dots+\alpha_{n-1} \hat{P}_{x_n}x_ {n-1}=0,
\]
one has that $\alpha_{2}$ must coincide with zero. It is clear that a repetition of this reasoning 
yields  that, for all $i$ in the set $\{1, \dots, n-1\}$, the scalar  $\alpha_{i}$ must coincide with  zero and, therefore,  the set $\mathfrak{X}$ 
 is linearly independent. Moreover, since the set $\{y_i:i=1, \dots , n\}$ is   linearly independent, and thus also the set $\{y_i:i=1, \dots , n\}$ is   linearly independent, it follows that  the operator 
 
 $$
 T_1=\sum_{i=1}^{n-1}(\hat{P}_{x_n}x_i)\otimes y_i
 $$
  has rank $n-1$ and lies in $\LLL$.

Analogously, one obtains  the equality 
$$ 
\left[ \hat{P}_{x_{n-1}},T_1 \right]  =T_1- \sum_{i=1}^{n-2} (\hat{P}_{x_{n-1}}x_i)\otimes y_i 
$$
for the operator $T_1$ and a similar reasoning yields that the rank $n-2$ operator 
$$
T_2=\sum_{i=1}^{n-2} (\hat{P}_{x_{n-1}}x_i)\otimes y_i
$$ lies in $\LLL$. 
Repeating this procedure as many times as required, one has  that 
$$
T_{n-1}=(\hat{P}_{x_{2}}x_1)\otimes y_1
$$ lies in $\LLL$. Hence, by Theorem \ref{corner} and observing that, for all $i=2,\dots, n$, 
$$
\hat{P}_{\hat{P}_{x_{i}}x_1}=\hat{P}_{x_{1}},
$$
it follows  that the operator $x_1\otimes y_1$ and also all the  operators $(\hat{P}_{x_{i}})x_1\otimes y_1$ lie in  the ideal $\LLL$. 

Back substitution in the equality  $$
T_{n-2}=(\hat{P}_{x_{3}}x_1)\otimes y_1+(\hat{P}_{x_{3}}x_2)\otimes y_2,
$$
similarly yields that, 
for all $i=3,\dots, n$, the operator $(\hat{P}_{x_{3}}x_2)\otimes y_2$ and also the operator $x_2\otimes y_2$ lie in $\LLL$. Clearly, the proof is complete after repeating this reasoning sufficiently many times.
\end{proof}

%%%%%%%%%%%%%%%%%%%%%%%%%%%%%%%%%%%%%%%
%                             finite rank aux2
%%%%%%%%%%%%%%%%%%%%%%%%%%%%%%%%%%%%%%%
 \begin{lemma}\label{finite rank aux2}
Let $\AAA$ be a continuous nest algebra, let $\LLL$ be a norm closed Lie ideal of $\AAA$ and let
$$
T=\sum_{i=1}^{n}x_i\otimes y_i,
$$
be a rank $n$ operator  lying in $\LLL$, with $n\geq 1$, where, for all $i, j=1, \dots,n$, the operator $x_i\otimes y_i$ lies in $\AAA$ and $\hat{P}_{x_i} =\hat{P}_{x_j}$.
Then, $T$ can be written as  a finite sum of rank one operators lying  in $\LLL$.

\end{lemma}
 \begin{proof}The proof will be carried out for $n>1$, since the result trivially holds for $n=1$.  

Define, for all $P\in \NNN$, the set $\mathfrak{X}_P$ by
$$
\mathfrak{X}_P=\{Px_1, Px_2, \dots, Px_n\},
$$and define the projection $Q$ by 
$$ 
Q=\bigwedge \{P\in\NNN: \mathfrak{X}_P \, \,   \text{is linearly independent}\}.
$$
Observe that, since the operator $T$ has rank $n$, the sets $\{x_i:i=1, \dots , n\}$ and $\{y_i:i=1, \dots , n\}$ are linearly independent.

Depending on the situation, it may happen that either $\hat{P}_{x_1}=Q$ or  $\hat{P}_{x_1}<Q$. This proof is accordingly  divided in the two parts a) and b).
\begin{listrom}[a)]
\item If $\hat{P}_{x_1}=Q$, then there exists a decreasing net $(Q_j)$ of projections in the nest that converges to $\hat{P}_{x_1}$ in the order topology 
and such that, for all $j$, the set $\{Q_jx_i:i=1,  \dots, n\}$ is linearly independent.  

 Let $k$ be a positive integer lying in the set $ \{i=1,  \dots, n\}$, let
 $$
\SSS_k=\sspan \{Q_jx_i:i=1,  \dots, k-1,k+1, \dots,  n\}
$$
  and let
 
 $$
 Q_j(\HHH)=\SSS_k \oplus {\SSS_k}^\perp
 $$
 be the decomposition of the Hilbert space $Q_j(\HHH)$ into the direct sum of $\SSS$ and its orthogonal complement $\SSS^\perp$ in the space $Q_j(\HHH)$. 
 Let 
$$
Q_jx_k={(Q_jx_k)}_p+{(Q_jx_k)}_o
$$
be the orthogonal decomposition of $Q_jx_k$ in $Q_j(\HHH)$ relatively to the direct sum above. 
Let  $w_j$ and $z_j$ be the elements of the Hilbert space  $\HHH$ defined by
$$
w_j=Q_j^{\perp}x_k \,\, , \quad \quad \quad z_j=\frac{1}{\|{(Q_jx_k)}_o\|^2}{(Q_jx_k)}_o
$$
and consider the  
 operator   $w_j\otimes z_j$. Clearly, the operator $w_j\otimes z_j$ lies in the nest algebra $\AAA$ and, therefore,
the operator 
\begin{align*}
\left[T, w_j\otimes z_j\right]&=\sum_{i=1}^n \intern{z_j,x_i}(w_j\otimes y_i)-\sum_{i=1}^n \intern{y_i,w_j}(x_i\otimes z_j)\\
&=\sum_{i=1}^n \frac{1}{\|{(Q_jx_k)}_o\|^2}\intern{{(Q_jx_k)}_o,x_i}\bigl((Q_j^{\perp}x_k)\otimes y_i\bigr)\\
&=\sum_{i=1}^n \frac{1}{\|{(Q_jx_k)}_o\|^2}\intern{{(Q_jx_k)}_o,Q_jx_i}\bigl((Q_j^{\perp}x_k)\otimes y_i\bigr)\\
&=\frac{1}{\|{(Q_jx_k)}_o\|^2}\intern{{(Q_jx_k)}_o,Q_jx_k}\bigl((Q_j^{\perp}x_k)\otimes y_k\bigr)\\
&=\frac{1}{\|{(Q_jx_k)}_o\|^2}\intern{{(Q_jx_k)}_o,{(Q_jx_k)}_o}\bigl((Q_j^{\perp}x_k)\otimes y_k\bigr)\\
&=(Q_j^{\perp}x_k)\otimes y_k
\end{align*}
 lies in $\LLL$. 
Since the net $(Q_j)$ converges to $\hat{P}_{x_1}$ in the order topology, by \cite{Er}, Lemma 1, the net $(Q_j x_k)$ converges $ \hat{P}_{x_1}x_k$ in the norm topology. Observing that $ \hat{P}_{x_1}x_k$ coincides with zero, it follows that the net  $(Q_j^{\perp}x_k)$ converges to
%\longrightarrow 
$x_k$ and, hence, 
$$
(Q_j^{\perp}x_k)\otimes y_k\longrightarrow x_k\otimes y_k.
$$
Since the Lie ideal $\LLL$ is norm closed, one has that the operator $x_k\otimes y_k$ lies in $\LLL$, thus ending the proof of part a).

\item Suppose now that  the projection $\hat{P}_{x_1}$ is less than the projection $Q$.   Let $(P_j)$ be an increasing net converging to $Q$ in the order topology and  
let $\XXX_j$ be the set defined, for all $j$, by 
$$
\XXX_j=\{P_j{x_i}:i=1,  \dots, n\}.
$$

 By the definition of the projection $Q$, all the sets $\XXX_j$ are  linearly dependent or, equivalently,  for all $j$, the Grammian determinant $\Delta_j$, defined by 
    $$
\Delta_j = \det \begin{bmatrix}
\intern{P_jx_{i}, P_jx_{k}}
\end{bmatrix}_{i,k=1,\dots,n}
$$
 coincides with zero. 
  Since, by \cite{Er}, Lemma 1, for all $i=1,  \dots, n$, the net $(P_j{x_i})$ 
   converges to ${Qx_i}$ in the norm topology,   the continuity of the determinant function 
   yields that the net $(\Delta_j)$ converges to the Grammian determinant $\Delta$ defined by
    $$
\Delta = \det \begin{bmatrix}
\intern{Qx_{i}, Qx_{k}}
\end{bmatrix}_{i,k=1,\dots,n}.
$$
Therefore, the determinant $\Delta$ coincides with zero and the set
$\{ Q x_i:i= 1,\dots, n\}
$  is linearly dependent. 

It is now clear that the projection $Q$ cannot be equal to the identity $I$. In fact, if   $Q$ were equal to $I$, then the set  $\{ x_i:i=1,  \dots, n\}$ would be a  linearly dependent set, yielding a contradiction. 

Let $(Q_j)$ be a decreasing net  of projections in the nest converging to $Q$ in the order topology
and such that, for all $j$, the set $\{Q_jx_i:i=1,  \dots, n\}$ is linearly independent.

Let $k$ be a positive integer lying in the set $ \{i=1,  \dots, n\}$ and, similarly to part a) of this proof, let
 $$
\SSS_k=\sspan \{Q_jx_i:i=1,  \dots, k-1,k+1, \dots,  n\},
$$
let
 
 $$
 Q_j(\HHH)=\SSS_k \oplus {\SSS_k}^\perp
 $$ 
 and let 
$$
Q_jx_k={(Q_jx_k)}_p+{(Q_jx_k)}_o
$$
be the orthogonal decomposition of $Q_jx_k$ in $Q_j(\HHH)$ relatively to the direct sum above. Let  $w_j$ and $z_j$ be the elements of the Hilbert space  $\HHH$ defined by
$$
w_j=Q_j^{\perp}x_k \,\, , \quad \quad \quad z_j=\frac{1}{\|{(Q_jx_k)}_o\|^2}{(Q_jx_k)}_o
$$
and consider the operator   $w_j\otimes z_j$ in the nest algebra $\AAA$.
Similar computations to those of part a), yield that the operator $\left[T, w_j\otimes z_j\right]$ lies in the Lie ideal $\LLL$  and satisfies the equality 
\begin{align*}
\left[T, w_j\otimes z_j\right] 
&=(Q_j^{\perp}x_k)\otimes y_k.
\end{align*}

Since the net $(Q_j^{\perp})$ converges to $Q^{\perp}$ in the order topology, by \cite{Er}, Lemma 1, the net $(Q_j^{\perp} x_k)$ converges $Q^{\perp}x_k$ in the norm topology. Hence, it follows that, for all $k$ in the set $\{1, \dots, n\}$, 
 the operator $(Q^{\perp}x_k) \otimes y_k$ lies in $\LLL$, yielding that the operator $TQ^\perp$ can be written as a finite sum of rank one operators lying in $\LLL$.  
Consequently, $TQ^\perp$ and also $TQ$ lie in $\LLL$. 

 Let $n_1$ be the rank of $TQ$, and observe that $n_1<n$ since the set $\{Qx_i:i=1,2,\dots, n\}$ is linearly dependent. In fact, 
 \begin{align*}
 TQ&=\sum_{i=1}^{n}Qx_i \otimes y_i\\
 &=\sum_{r=1}^{n_1} x_{i_r}\otimes y^{\prime}_{i_r}, 
\end{align*}
 where the mapping $r\mapsto i_r$, from $\{1, \dots, n_1\}$ into $\{1, \dots, n\}$ is injective and, for all indices $r$,       
  the set $\{y^{\prime}_{i_r}: r=1, \dots , n_1\}$ is linearly independent.  Moreover, since each  $y^{\prime}_{i_r}$ is a linear combination of the elements of the set $\{y_i:i=1,\dots,n\}$, all the operators $x_{i_r}\otimes y^{\prime}_{i_r}$ lie in the nest algebra $\AAA$ and, consequently, the operator $TQ$ satisfies all the conditions of this lemma. 
  
Now it may happen that $TQ$ satisfies the condition of part a), in which case it follows immediately that $T$ is a finite sum of rank one operators lying in $\LLL$. Alternatively, if $TQ$ satisfies the condition of part b), then 
$$
TQ=TQQ_1+TQQ_1^{\perp},
$$
where $Q_1$ is the projection in the nest defined by
$$
Q_1=\bigwedge \{P\in\NNN: \mathfrak{X}_{1P} \, \,   \text{is linearly independent}\}
$$
and
$$
\mathfrak{X}_{1P}=\{Px_{i_r}: r=1,\dots, n_1\}.
$$ 
Applying the reasoning of part b) of the proof, either $TQQ_1$ satisfies the conditions in part a), thus ending the proof, or is an operator of rank $n_2$ less than $n_1$ such that 
$$
TQQ_1=TQQ_1Q_2+TQQ_1Q_2^{\perp},
$$
where $Q_2$ is a projection defined analogously to $Q_1$ above. 
 Repeating this procedure a sufficient (finite) number of times, either situation a) occurs at some point or 
 situation b) always happens. In the first hypothesis, the proof ends and in the second hypothesis, since the rank of the relevant operators strictly decreases in each step, it is eventually possible  to obtain a  rank one operator $TQQ_1Q_2Q_3\dots Q_l$ lying in $\LLL$, thus concluding the proof.
  \end{listrom}
 \end{proof}

%%%%%%%%%%% finite rank theorem %%%%%%%%
 \begin{theorem}\label{finite rank}
Let $\AAA$ be a continuous nest algebra, let $\LLL$ be a norm closed Lie ideal of $\AAA$ and let  $T$ be a  finite rank operator  
  in $\LLL$. Then $T$ can be written as a finite sum of rank one operators lying in $\LLL$.
 \end{theorem}
 %%%%%%%%%%% finite rank theorem %%%%%%%%

 \begin{proof}
 Clearly, the assertion holds if  $T=0$ or if $T$ is a rank one operator. 
 Let $T$ be an operator of rank $n\geq2$ lying in the  Lie Ideal $\LLL$.  
It is possible to write the operator $T$ as   

 $$
 T=\sum_{i=1}^{n}x_i\otimes y_i,
 $$
 where, for all $i=1,\dots, n$, the rank one operator $x_i\otimes y_i$ lies in $\AAA$ (cf. \cite{KD,Er}).
 Suppose, without loss of generality, that,   for all indices $i, j=1,\dots, n$, 
 
 $$
 i\leq j \, \, \Rightarrow \, \, \hat{P}_{x_i} \leq\hat{P}_{x_j}.
 $$
 
   This proof is divided  into part a) and part b)
  according to $\hat{P}_{x_1} <\hat{P}_{x_2}$ or $\hat{P}_{x_1} =\hat{P}_{x_2}$, respectively.

  \begin{listrom}[a)]
  \item Suppose that $\hat{P}_{x_1} <\hat{P}_{x_2}$. Either, for all indices $i, j=1,\dots, n$,
   
  $$
 i< j \, \, \Rightarrow \, \, \hat{P}_{x_i} <\hat{P}_{x_j},
 $$ 
     and the result immediately follows from Lemma \ref{finite rank aux1}, or there exists a positive integer $k<n$  such that
      $$
 \hat{P}_{x_1} <\dots <\hat{P}_{x_k}=\hat{P}_{x_{k+1}}.
 $$ 
  If this is the case, then the operator
 \begin{align*}
  T_{k-1}&=T-[\hat{P}_{x_{k}},T]\\
 &=T -\hat{P}_{x_{k}}T\hat{P}_{x_{k}}^{\perp}\\
 &=T -T\hat{P}_{x_{k}}^{\perp}
\end{align*}  lies in $\LLL$. Since,
 for all indices $i, j=1,\dots, k-1$, 
 
 $$
 i\leq j \, \, \Rightarrow \, \, \hat{P}_{x_i} <\hat{P}_{x_j}.
 $$
 the operator  
 $$
  T_{k-1}=\sum_{i=1}^{k-1}(\hat{P}_{x_k}x_i)\otimes y_i,
 $$is a rank $k-1$ operator,   and by Lemma \ref{finite rank aux1}, it follows that,  for all $i=1,\dots, k-1$,
 the operator $(\hat{P}_{x_k}x_i)\otimes y_i$ lies in $\LLL$. 
 Hence, by Theorem \ref{corner}, all the operators $x_1\otimes y_1, \dots, x_{k-1}\otimes y_{k-1}$ lie in the Lie ideal $\LLL$ and, therefore,  
 $$
 S=x_k\otimes y_k + x_{k+1}\otimes y_{k+1}+ \dots +x_{n}\otimes y_{n}
 $$
  is also an element of the Lie ideal $\LLL$.

\item  Suppose now that $\hat{P}_{x_1} =\hat{P}_{x_2}$. Either 
$$
\hat{P}_{x_1} =\hat{P}_{x_2}=\dots=\hat{P}_{x_n} 
$$
  and the result immediately follows from Lemma \ref{finite rank aux2}, or there exists a positive integer $m<n$  such that
 $$
 \hat{P}_{x_1} =\hat{P}_{x_2}=\dots =\hat{P}_{x_m}<\hat{P}_{x_{m+1}}.
 $$
 Hence, the operator  
  \begin{align*}
  T_{m}&=T-[\hat{P}_{x_{m+1}},T]\\
 &=T-\hat{P}_{x_{m+1}}T\hat{P}_{x_{m+1}}^{\perp}\\
&=T-T\hat{P}_{x_{m+1}}^{\perp}\\
 &=\sum_{i=1}^{m}(\hat{P}_{x_{m+1}}x_i)\otimes y_i,
\end{align*} 
 lies in $\LLL$. If the rank of $T_m$ equals $m$, by Lemma \ref{finite rank aux2}, the operator $T_m$ can be written as a finite sum of rank  one operators lying in $\LLL$.  By Theorem \ref{corner} and the proof of Lemma \ref{finite rank aux2}, for all $i=1,2,\dots, m$, the operator $({\hat{P}_{x_{m+1}}}^{\perp}x_i)\otimes y_i$ lies in $\LLL$. In fact, the operator $T_m$ either satisfies condition a) or condition b) in the proof of Lemma \ref{finite rank aux2}. In the first case, it was shown that, for all $i=1,2,\dots, m$, the operator $({\hat{P}_{x_{m+1}}} x_i)\otimes y_i$ lies in $\LLL$ and, consequently, by Theorem \ref{corner}, also the operator $({\hat{P}_{x_{m+1}}}^{\perp}x_i)\otimes y_i$ lies in $\LLL$.
 
 If, on the other hand, the operator $T_m$ satisfies condition b), the proof of Lemma \ref{finite rank aux2} shows that there exists a projection $Q$ in the nest with 
 \[
 {\hat{P}_{x_{m}}} < Q < {\hat{P}_{x_{m+1}}}, \quad \quad \quad \quad \hat{P}_{Q^\perp x_i}< {\hat{P}_{x_{m+1}}}
 \]
and such that, for all  $i=1,2,\dots, m$, the rank one operator $(Q^\perp x_i)\otimes y_i$ lies in $\LLL$. Theorem \ref{corner} now guarantees that,  for all $i=1,2,\dots, m$, the operator $({\hat{P}_{x_{m+1}}}^{\perp}x_i)\otimes y_i$ lies in $\LLL$. Observe that a similar reasoning to that used to show that $Q$ cannot be equal to $I$ can also be used to see that $Q$ is different from ${\hat{P}_{x_{m+1}}}$.

 It is now possible to conclude that  the equality

\begin{align}\label{equation1}
T&=T_m+\sum_{i=1}^m({\hat{P}_{x_{m+1}}}^{\perp}x_i)\otimes y_i+\sum_{i=m+1}^nx_i\otimes y_i,
\end{align}
yields that the operator $S$ defined by
$$
S=\sum_{i=m+1}^nx_i\otimes y_i
$$
 lies in the Lie ideal $\LLL$.

 In the case of  the rank of  $T_m$ being less than $m$, the set $\{\hat{P}_{x_{m+1}}x_i:i=1, \dots,m\}$ is linearly dependent. There exists, nevertheless,  a subset $J$ of $\{i=1, \dots,m\}$ such that $\{\hat{P}_{x_{m+1}}x_i:i\in J\}$ is a maximal linearly independent subset of $\{\hat{P}_{x_{m+1}}x_i:i=1, \dots,m\}$ and therefore, for all $k\in J_d$, where 
 $$
 J_d=\{1, \dots,m\}\backslash J,
 $$
 the identity
 $$
 \hat{P}_{x_{m+1}}x_k=\sum_{j\in J} a_{kj}\hat{P}_{x_{m+1}}x_j
 $$
 holds.
 Consequently, the operator $T_m$ can be re-written as 

 \begin{align*}
T_m &=\sum_{i\in J}(\hat{P}_{x_{m+1}}x_i)\otimes y_i+\sum_{k\in J_d}(\hat{P}_{x_{m+1}}x_k)\otimes y_k\\
 &=\sum_{i\in J}(\hat{P}_{x_{m+1}}x_i)\otimes y_i+\sum_{k\in J_d}\Bigl(\sum_{j\in J} (a_{kj}\hat{P}_{x_{m+1}}x_j)\Bigr)\otimes y_k\\
 &=\sum_{i\in J}(\hat{P}_{x_{m+1}}x_i)\otimes y_i+\sum_{j\in J} \Bigl(\sum_{k\in J_d}(a_{kj}\hat{P}_{x_{m+1}}x_j)\otimes y_k\Bigr)\\
  &=\sum_{i\in J}(\hat{P}_{x_{m+1}}x_i)\otimes y_i+\sum_{j\in J} \Bigl(\sum_{k\in J_d}(\hat{P}_{x_{m+1}}x_j)\otimes \bar{a}_{kj}y_k\Bigr)\\
 &=\sum_{i\in J}(\hat{P}_{x_{m+1}}x_i)\otimes y_i+\sum_{j\in J} \Bigl((\hat{P}_{x_{m+1}}x_j)\otimes \sum_{k\in J_d}\bar{a}_{kj}y_k\Bigr)\\
 &=\sum_{i\in J}(\hat{P}_{x_{m+1}}x_i)\otimes \Bigl(y_i+\sum_{k\in J_d}\bar{a}_{ki}y_k\Bigr)=\sum_{i\in J} (\hat{P}_{x_{m+1}}x_i)\otimes y^{\prime}_i,
%\\
\end{align*}
where the sets $\{\hat{P}_{x_{m+1}}x_i: i\in J\}$ and $\{y^{\prime}_i: i\in J\}$ are linearly independent.
  By Lemma \ref{finite rank aux2}, the operator $T_m$ can be written as a finite sum of rank  one operators lying in $\LLL$.  
 Observe also that, by Theorem \ref{corner} and the proof of Lemma \ref{finite rank aux2}, for all $i\in J$, the operator $({\hat{P}_{x_{m+1}}}^{\perp}x_i)\otimes y^{\prime}_i$ lies in $\LLL$.

It follows from the equality (\ref{equation1}) 
that the operator 
\[
S_m=\sum_{i=1}^m({\hat{P}_{x_{m+1}}}^{\perp}x_i)\otimes y_i+\sum_{i=m+1}^nx_i\otimes y_i.
\]
lies in the ideal.

It is now necessary to show that $S_m$ can be written as a finite sum of rank one operators in $\LLL$.    
 Since 
 
  \begin{align*} 
S_m&=\sum_{i\in J}({\hat{P}_{x_{m+1}}}^{\perp}x_i)\otimes y_i+\sum_{i\in J_d}({\hat{P}_{x_{m+1}}}^{\perp}x_i)\otimes y_i+\sum_{i=m+1}^nx_i\otimes y_i\\
&=\begin{multlined}[t]\sum_{i\in J}({\hat{P}_{x_{m+1}}}^{\perp}x_i)\otimes \Bigl(y_i+\sum_{k\in J_d}\bar{a}_{ki}y_k \Bigr)- \sum_{i\in J}(\hat{P}^{\perp}_{x_{m+1}}x_i)\otimes  \Bigl(\sum_{k\in J_d}\bar{a}_{ki}y_k \Bigr)+\\ 
+\sum_{i\in J_d}({\hat{P}_{x_{m+1}}}^{\perp}x_i)\otimes y_i+\sum_{i=m+1}^nx_i\otimes y_i\end{multlined}\\
&=\begin{multlined}[t]\sum_{i\in J}({\hat{P}_{x_{m+1}}}^{\perp}x_i)\otimes y^{\prime}_i -\sum_{k\in J_d}  \Bigl(\sum_{i\in J}({a}_{ki}\hat{P}^{\perp}_{x_{m+1}}x_i)\otimes y_k \Bigr)+\\
 +\sum_{i\in J_d}({\hat{P}_{x_{m+1}}}^{\perp}x_i)\otimes y_i
 +\sum_{i=m+1}^nx_i\otimes y_i
 \end{multlined}\\
 &=\begin{multlined}[t]\sum_{i\in J}({\hat{P}_{x_{m+1}}}^{\perp}x_i)\otimes y^{\prime}_i +\sum_{k\in J_d}  \Bigl(\sum_{i\in J} \Bigl({\hat{P}_{x_{m+1}}}^{\perp}x_k-{a}_{ki}\hat{P}^{\perp}_{x_{m+1}}x_i \Bigr) \Bigr)\otimes y_k+\\
 +\sum_{i=m+1}^nx_i\otimes y_i,
  \end{multlined}\\
\end{align*} 
where the first summand is a sum of rank one operators in $\LLL$, it follows that the operator 
$$
S=\sum_{k\in J_d}  \Bigl(\sum_{i\in J} \Bigl({\hat{P}_{x_{m+1}}}^{\perp}x_k-{a}_{ki}\hat{P}^{\perp}_{x_{m+1}}x_i \Bigr) \Bigr)\otimes y_k+\sum_{i=m+1}^nx_i\otimes y_i
$$
lies in $\LLL$ and has rank $l$ less than $n$.  The operator $S$ can be re-written as a sum of $l$ rank one operators lying in $\AAA$ (cf.  \cite{KD, Er}). 
\end{listrom}

 Applying again the same reasoning to either of the operators $S$ of part a) or part b), according to the case, it is possible to conclude, after a finite number of steps, that $T$ can be written as a finite sum of rank one operators in $\LLL$.
 \end{proof} 
 
\section{A characterisation theorem}\label{character}
Recall that a mapping $\varphi:\NNN \rightarrow \NNN$, defined on a nest $\NNN$, is called a \emph{homomorphism} if, for all projections $P$ and $Q$ in $\NNN$,

$$
P\leq Q \quad \Longrightarrow \quad \varphi(P)\leq \varphi(Q).
$$
A homomorphism $\varphi$ is said to 
be \textita{left order continuous} if, for all subsets $\MMM $ of 
the nest $\NNN $,  the projection $\varphi (\bigvee \MMM ) $ is equal to the 
supremum $\bigvee \varphi(\MMM)$. 

\begin{proposition}\label{order homomorphism}Let $\AAA$ be a continuous nest algebra associated to a   nest $\NNN$ and let $\LLL$ be a norm closed Lie ideal of $\AAA$. Then, the mapping $P\mapsto P^{\prime}$  defined, for all projections $P$ in the nest $\NNN$, by
\begin{align}\label{hom}
P^{\prime}=\bigvee \left\{P_y\in \NNN: x\otimes y \in \LLL\, \,  \wedge \, \, \hat{P}_x< P\right\}
\end{align}
 is a left order continuous homomorphism on $\NNN$. 
\end{proposition} 

\begin{proof}Let $P$ and $Q$ be projections in the nest $\NNN $ such that $   
P < Q   $. It will be shown that   the projection 
 $P^{\prime } $ is less than or equal to the projection $ Q^{\prime }    $.

If the projection $Q^{\prime }$ were less than the projection $P^{\prime }   $, then there would exist an operator $x\otimes y$ 
 in $\LLL   $ 
such that 
$$
{\hat{P } }_x < P   \, , \quad \quad \quad   Q^{\prime }   <{P 
}_{y  }.
$$ 
Therefore, by Theorem \ref{corner}, all operators $   w\otimes z$, such that 
$ {\hat{P } }_{x } \leq {\hat{P } }_{w  }  $
 and ${P }_{z  } \leq {P }_{y  }  $, would lie in  $\LLL   $. 
Hence,  
\[  Q^{\prime }< \bigvee \left\{ {P }_{z  }  \in \NNN  :   w\otimes z 
\in \LLL ,\,\,  {\hat{P } }_{w  } < Q \right\}   \text{ ,}\]
yielding a contradiction, and thus concluding the proof that the mapping is a homomorphism.

To show  that the order homomorphism $P\mapsto P^\prime$  is left order continuous on $\NNN$, let  $\MMM $ be a subset    of the nest $\NNN $, let $ {\MMM }^{\prime }$ be the set defined by
 $$
  {\MMM }^{\prime }=\bigl\{  P ^{\prime }   \in \NNN :   P\in \MMM   \bigr\}
 $$
 and let the projection $Q$ be the supremum 
of $\MMM $.

If the projection $Q$ lies in $\MMM $, then, since the mapping $P\mapsto
P^{\prime }$ is an order homomorphism,   the projection $Q^{\prime}$ 
coincides with the supremum $ \bigvee {\MMM }^{\prime
}$.

Suppose now that the supremum $Q$ does not lie in $\MMM $.    
It follows that
\begin{align*}
Q^{\prime }&= 
\bigvee \Bigl\{P_y\in \NNN: x\otimes y \in \LLL\, \,  \wedge \, \, \hat{P}_x< Q\Bigr\}\\
&= \bigvee \biggl(\bigcup _{P\in \NNN , P<Q} \Bigl\{ {P }_{y  }  \in \NNN  :  x\otimes y
\in \LLL   ,\,\,  {\hat{P } }_{x  }    <   P  \Big\} \biggr).
\end{align*} 

By Theorem \ref{corner} and  the fact that the mapping $P\mapsto P^{\prime }$ is an order
homomorphism, 
\begin{align*}
    Q^{\prime }&= \bigvee \biggl(\bigcup _{P\in \MMM } \Bigl\{ {P }_{y  }   \in \NNN :   
    x\otimes y  \in \LLL    ,\,\,  {\hat{P } }_{x  }    <   P  \Bigr\} 
\biggr)\\
                            %& &                  \\
          &= \bigvee  \bigl\{  P ^{\prime }   \in \NNN :   P\in \MMM   \bigr\} \\
                            %& &
                            &= \bigvee {\MMM }^{\prime } \text{ ,}
 \end{align*}
which concludes the proof.
\end{proof}

\begin{Lemma}\label{homomorphism rank 1}Let $\AAA$ be a continuous nest algebra associated to a   nest $\NNN$, let $\LLL$ be a norm closed Lie ideal of $\AAA$ and let the mapping $P\mapsto
P^{\prime }$ be the left order continuous homomorphism {\rm{(\ref{hom})}}. Then, a rank one operator $x\otimes y$ lies in $\LLL$ if and only if, for all projections $P$ in the nest $\NNN$,
$${P^{\prime}}^\perp(x\otimes y)P=0.$$ 
\end{Lemma} 
It should be observed that any rank one operator satisfying the above condition must lie in the nest algebra $\AAA$.
\begin{proof}If   $x\otimes y$ is a rank one operator  in $ \LLL $ then, for all projections $P$ in the nest, 

 \[
{P^{\prime}}^\perp(x\otimes y)P=(Px)\otimes {P^{\prime}}^\perp y
\]
coincides with zero. 
In fact, either $P\leq \hat{P}_x$, which leads to $Px$ being equal to zero, or $ \hat{P}_x< P$, which implies, by the definition of the order homomorphism $P\mapsto P^{\prime}$, that $P_y \leq P^{\prime}$ and therefore ${P^{\prime}}^\perp y$ coincides with zero.

Conversely, let $x\otimes y$ be  such that, for all projections $P$  in $\NNN$, the operator ${P^{\prime}}^\perp (x\otimes y)P$ 
coincides with zero.
Observe that, for any projection $P$ greater than $ \hat{P}_x$, 
 the element ${P^{\prime}}^\perp y$ must be equal to zero, i.e.,  
 $P_y\leq P^{\prime}$. 
 
 Notice also that   a rank one operator $w\otimes z$ such that $P\leq \hat{P}_w$ and  $P_z\leq P^{\prime}$ must lie in the norm closed ideal $\LLL$. Clearly, 
 by Theorem \ref{corner} and the definition of the mapping $P\mapsto P^{\prime}$, all operators $w\otimes z$, with $P\leq \hat{P}_w$ and   $P_z<P^{\prime}$, lie in $\LLL$. If $w\otimes z$ is such that $P\leq \hat{P}_w$ and   $P_z=P^{\prime}$, then there exists a sequence $(z_n)$ in 
 $$
\bigcup_{Q\in \NNN, Q< P^{\prime}}Q(\HHH), 
$$
 converging to $z$ in the norm topology. Hence $(w\otimes z_n)$ is a sequence in $\LLL$ whose limit $w\otimes z$ also   lies in $\LLL$.

It is possible to find a sequence $(x_n\otimes y)$  converging to $x\otimes y$ and such that, for all $n$, the projection $\hat{P}_{x}$ is less than the projection $\hat{P}_{x_n}$. Since, by the reasoning above, this sequence  lies in $\LLL$, it follows that $x\otimes y$ is an element of the ideal.
\end{proof}

 \begin{theorem}\label{homomorphism finite rank}
Let $\AAA$ be a continuous nest algebra associated to a   nest $\NNN$, let $\LLL$ be a norm closed Lie ideal of $\AAA$ and let the mapping $P\mapsto
P^{\prime }$ be the left order homomorphism {\rm{(\ref{hom})}}. 
Then a finite rank operator $T$ in the nest algebra $\AAA$ lies in $\LLL$ if and only if, for all projections $P$ in $\NNN$,
 $$
 {P^{\prime}}^{\perp}TP=0.
 $$
 \end{theorem}
 
 \begin{proof} The assertion trivially holds for $T=0$. It will be assumed from now on that $T$ is an operator of  rank $n\geq 1$ in the nest algebra $\AAA$.

 If the operator $T$ lies in $\LLL$, then by Theorem \ref{finite rank},  it can be written as  a finite sum of rank one operators  in the ideal. By Lemma \ref{homomorphism rank 1}, each of this rank one operators satisfies the condition associated to the homomorphism $P\mapsto P^{\prime}$ and, therefore, also $T$  satisfies the condition.
 
 Conversely, suppose that $T$ is an  operator  in the nest algebra with rank $n\geq 1$ and  such that, for all $P$ in $\NNN$, 
 $$
 {P^{\prime}}^{\perp}TP=0.
 $$
As a consequence, the  operator $T$ lies in the norm closed (associative) ideal 
$$
\mathfrak{B}=\{S\in \AAA: {P^{\prime}}^{\perp}SP=0\}
$$
of the nest algebra $\AAA$ and thus, by Theorem \ref{finite rank}, there exist finitely many  operators $x_i\otimes y_i$ in $\mathfrak{B}$ such that 
$$
T=\sum_{i}x_i\otimes y_i.
$$
Therefore, by Lemma \ref{homomorphism rank 1}, for all $i$, the operator 
$x_i\otimes y_i$ lies in $\LLL$ and thus also $T$ is an operator in this ideal.
\end{proof} 

\subsection*{Acknowledgments} The author wishes to thank I. G. Todorov for his  helpful remarks during  the preparation of the manuscript. The financial support of CEAF, through a  FCT grant under the Research Units Pluriannual Funding Programme, is also gratefully acknowledged.


\begin{thebibliography}{1}

   \bibitem{B} S. Bradley, \textit{Finite rank operators in certain algebras}, Canad. Math.  Bull. (4) \textbf{42} (1999), 452--462.



 \bibitem{KD} K.  R. Davidson, \textit{Nest Algebras}, Longman, 1988.
\bibitem{DDH} K.  R. Davidson, A. P. Donsig and T. D.  Hudson, \textit{Norm-closed bimodules of nest algebras}, J.  Operator  Theory  \textbf{39} (1998), 59--87.

   \bibitem{DL} Z. Dong and Shijie Lu, \textit{Finite rank operators in closed maximal triangular algebras II}, Proc.  Amer. Math. Soc. (5) \textbf{131} (2000), 167--172.

   \bibitem{D} Z. Dong, \textit{Finite rank operators in Jacobson Radical $\mathcal{R}_{\NNN \otimes \MMM}$}, Czechoslovak Math. J. (56) \textbf{131} (2006), 287--298.

   \bibitem{Er} J. A. Erdos, \textit{Operators of finite rank in nest algebras}, J.  London Math.  Soc. \textbf{43} (1968), 391--397.
   
   \bibitem{EP} J. A. Erdos and S. C. Power, \textit{Weakly closed ideals of nest algebras}, J.  Operator  Theory (2) \textbf{7} (1982), 219--235.
   
   
   \bibitem{HP} A. Hopenwasser and V. I. Paulsen, \textit{Lie ideals in  operator algebras}, J.   Operator Theory   (2) \textbf{52} (2004), 325--340.

\bibitem{HMS} T.  D. Hudson, L. W. Marcoux and A. R.  Sourour, \textit{Lie ideals in triangular operator algebras}, Trans.  Amer.  Math. Soc.  (12) \textbf{126} (1998), 3321--3339.

  \bibitem{MS} L. W. Marcoux and A. R.  Sourour, \textit{Conjugation-invariant subspaces and Lie ideals in non-selfadjoint operator algebras}, J. London   Math. Soc.   (2) \textbf{65} (2002), 493--512.

   \bibitem{LuYu} F. Lu and X. Yu, \textit{Lie and Jordan ideals in reflexive algebras}, Integr.  Equ. Oper.  Theory  \textbf{59} (2007), 189--206.



   \bibitem{LO} L.  Oliveira, \textit{Weak*-closed Jordan ideals of nest algebras}, Math.  Nachr. \textbf{248--249} (2003), 129--143.

  \bibitem{Ring} J.  R.  Ringrose, \textit{On some algebras of operators}, Proc.  London Math. Soc. \textbf{15} (1965), 61--83.


\end{thebibliography}
\end{document}